%% file: Chow.Pavone.ACC14.tex
\documentclass[letterpaper, 10 pt, conference]{ieeeconf}  

\IEEEoverridecommandlockouts                              

\usepackage{graphicx}
\usepackage{amssymb,amsmath}

\usepackage{amsthm}
\usepackage{comment,color}

\graphicspath{{./fig/}}

\usepackage[bookmarks=true]{hyperref}

\input{preamble}

\newcommand{\fil}{\mathcal F}
\newcommand{\cs}{\mathcal Z}
\newcommand{\csd}{\mathcal U}

\newcommand{\risk}{\rho}

\newcommand{\upol}{\mathcal U^{\mathrm{poly}}}
\newcommand{\upolv}{\mathcal U^{\mathrm{poly}, V}}

\renewcommand{\baselinestretch}{0.916}

\title{\LARGE \bf
A Framework for Time-Consistent, Risk-Averse \\Model Predictive Control: Theory and Algorithms }

\author{Yin-Lam Chow, Marco Pavone
\thanks{Y.-L. Chow is with the Institute for Computational and Mathematical Engineering, Stanford University, Stanford, CA 94305, USA. Email: {\tt ychow@stanford.edu}.}
\thanks{M. Pavone is with the Department of Aeronautics and Astronautics, Stanford University, Stanford, CA 94305, USA. Email: {\tt pavone@stanford.edu}.}
}

\begin{document}
\maketitle
\thispagestyle{empty}
\pagestyle{empty}

\begin{abstract}
In this paper we present a framework for risk-averse model predictive control (MPC) of linear systems affected by multiplicative uncertainty. Our key innovation is to consider time-consistent, dynamic risk metrics as objective functions to be minimized. This framework is axiomatically justified in terms of time-consistency of risk preferences, is amenable to dynamic optimization, and is unifying  in the sense that it captures a full range of risk assessments from risk-neutral to worst case. Within this framework, we propose and analyze an online risk-averse MPC algorithm that is provably stabilizing. Furthermore, by exploiting the dual representation of time-consistent, dynamic risk metrics, we cast the computation of the MPC control law as a convex optimization problem amenable to implementation on embedded systems. Simulation results are presented and discussed.
\end{abstract}

\section{Introduction}\label{sec:intro}
Model predictive control (MPC) is one of  the most popular methods to address optimal control  
problems in an online setting \cite{Qin_Badgwell_03, Wang.Boyd:CST10}. The key idea behind MPC is to obtain the control action by repeatedly solving, at each sampling instant,
a finite horizon open-loop optimal control problem, using the current state of the plant as the initial state; the result of the optimization is an 
(open-loop) control sequence, whose first element is applied to control the system \cite{Mayne.ea:Auto00}.

The classic MPC framework does not provide a systematic way to address  model uncertainties and disturbances \cite{Bernardini_Bemporad_12}. Accordingly, one of the main research thrusts for MPC is to find techniques to guarantee persistent feasibility and stability in the presence of disturbances. Essentially, current techniques fall into two categories: (1) min-max (or worst-case) formulations, where the performance indices to be minimized are computed with respect to the worst possible disturbance realization \cite{Kothare_Balakrishnan_Morari_96, Souza_06, Park_Kwon_02}, and (2) stochastic formulations, 
where \emph{risk-neutral expected} values of performance indices (and possibly constraints) are considered \cite{Bernardini_Bemporad_12, Primbs_Sung_09}.

The main drawback of the worst-case approach is that the control law may be too conservative, since the MPC law is required to guarantee stability and constraint fulfillment under the worst-case scenario. On the other hand, stochastic formulations whereby the assessment of future random outcomes is accomplished through a risk-neutral expectation may be unsuitable in scenarios where one desires to protect the system from large deviations. 

The objective of this paper is to introduce a systematic method to include risk-aversion in MPC. The inclusion of  risk aversion is important for several  reasons. First, in uncertain environments, a guaranteed-feasible solution may not exist and the issue becomes how to properly balance between planner conservatism and the risk of infeasibility (clearly this can not be achieved with a worst-case approach). Second, risk-aversion allows the control designer to increase policy robustness by limiting confidence in the model. Third, risk-aversion serves to prevent rare undesirable events. Finally, in a reinforcement learning framework when the world model is not accurately known, a risk-averse agent can cautiously balance exploration versus exploitation for fast convergence and to avoid ``bad" states that could potentially lead to a catastrophic failure \cite{Defourny_08, Moldovan:2012}.

Inclusion of risk-aversion in MPC is difficult for two main reasons. First,  it appears to be difficult to model risk in multi-period settings in a way that matches  intuition  \cite{Moldovan:2012}. In particular, a common strategy to include risk-aversion in multi-period contexts is to apply a \emph{static} risk metric, which assesses risk from the perspective of a single point in time,  to the total cost of the future stream of random outcomes. However, using static risk metrics  in multi-period decision problems  can lead to an over or under-estimation of the true dynamic risk, as well as to a potentially ``inconsistent" behavior, see \cite{Iancu_11} and references therein. Second, optimization problems involving risk metrics tend to be computationally intractable, as they do not allow a recursive estimation of risk. In practice, risk-averse MPC often resolves into the minimization of the expected value of an aggregate, risk-weighted objective function \cite{Zafra:11b,vanOverloop:08}.

In this paper, as a radical departure from traditional approaches, we leverage recent strides in the theory of \emph{dynamic} risk metrics developed by the operations research community \cite{rus_shapiro_06, rus_09} to include risk aversion in MPC. The key property of \emph{dynamic} risk metrics is that,  by assessing risk at multiple points in time, one can guarantee \emph{time-consistency} of risk preferences over time  \cite{rus_shapiro_06, rus_09}. In particular, the essential requirement for time consistency is that if a certain outcome is considered less risky in all states of the world at stage $k+1$, then it should also be considered less risky at stage $k$. Remarkably, in \cite{rus_09}, it is proven that any risk measure that is time consistent can be represented as a \emph{composition} of one-step risk metrics, in other words, in multi-period settings, risk (as expected) should be compounded over time.

The contribution of this paper is threefold. First, we introduce a notion of dynamic risk metric, referred to as Markov dynamic polytopic risk metric, that captures a full range of risk assessments and enjoys a geometrical structure that is particularly favorable from a computational standpoint. Second, we present and analyze a \emph{risk-averse} MPC algorithm that minimizes in a receding--horizon fashion a Markov dynamic polytopic risk metric, under the assumption that the system's model is linear and is affected by multiplicative uncertainty. Finally, by exploring the ``geometrical" structure of Markov dynamic polytopic risk metrics, we present a convex programming formulation for risk-averse MPC that is amenable to real-time implementation (for moderate horizon lengths). Our framework has three main advantages: (1) it is axiomatically justified, in the sense that risk, by construction, is assessed in a time-consistent fashion; (2) it is amenable to dynamic and convex optimization, thanks to the compositional form of Markov dynamic polytopic risk metrics and their  geometry; and (3) it is general, in the sense that it captures a full range of risk assessments from risk-neutral to worst case. In this respect, our formulation represents a \emph{unifying} approach for risk-averse MPC.

The rest of the paper is organized as follows. In Section \ref{sec:prelim} we provide a review of the theory of dynamic risk metrics. In Section \ref{sec:sys} we discuss the stochastic model we will consider in this paper. In Section \ref{sec:risk} we introduce and discuss the notion of Markov dynamic polytopic risk metrics. In Section \ref{sec:IHC} we state the infinite horizon optimal control problem we wish to address and in Section \ref{sec:stab} we derive conditions for risk-averse closed-loop stability.  In Section \ref{sec:MPC} and \ref{sec:alg} we present, respectively, a risk-averse model predictive control law and approaches for its computation. Numerical experiments are presented and discussed in Section \ref{sec:example}. Finally,  in Section \ref{sec:conclusion} we draw some conclusions and we discuss directions for future work.

%
%
%
%

\section{Review of Dynamic Risk Theory}\label{sec:prelim}
In this section, we briefly describe the theory of coherent and dynamic risk metrics, on which we will rely extensively later in the paper. The material presented in this section summarizes several novel results in risk theory achieved in the past 10 years. Our presentation strives to present this material in a intuitive fashion and with a notation tailored to control applications.

\subsection{Static, coherent measures of risk}\label{subsec:static_risk}
Consider a probability space $(\Omega, \fil, \probnoarg)$, where $\Omega$ is the set of outcomes (sample space), $\fil$ is a $\sigma$-algebra over $\Omega$ representing the set of events we are interested in, and $\probnoarg$ is a probability measure over $\fil$. In this paper we will focus on disturbance models characterized by probability \emph{mass} functions, hence we restrict our attention to finite probability spaces (i.e.,  $\Omega$ has a finite number of elements or, equivalently, $\fil$ is a finitely generated algebra). 
Denote with $\cs$ the space of random variables $Z:\Omega\mapsto (-\infty, \infty)$ defined over the probability space $(\Omega, \fil, \mathbb P)$. In this paper a random variable $Z\in \cs$ is interpreted as a cost, i.e., the smaller the realization of $Z$, the better. 
For $Z, W$, we denote by $Z\leq W$ the point-wise partial order, i.e., $Z(\omega)\leq W(\omega)$ for all $\omega\in \Omega$.

By a \emph{risk measure} (or \emph{risk metric}, we will use these terms interchangeably) we understand a function $\risk(Z)$ that maps an uncertain outcome $Z$ into the extended real line $\reals \cup\{ +\infty\}\cup \{-\infty\}$. In this paper we restrict our analysis to \emph{coherent risk measures}, defined as follows:

\begin{definition}[Coherent Risk Measures]\label{def:crm}
A coherent risk measure is a mapping $\risk:\cs \rightarrow \reals$, satisfying the following four axioms:
\begin{description}
\item[A1] Convexity: $\risk(\lambda Z + (1-\lambda)W)\leq \lambda\risk(Z) + (1-\lambda)\risk(W)$, for all $\lambda\in[0,1]$ and $Z,W \in\cs$;
\item[A2] Monotonicity:  if $Z\leq W$ and $Z,W \in\cs$,  then $\risk(Z)\leq\risk(W)$;
\item[A3] Translation invariance: if $a\in \reals$ and $Z\in \cs$, then $\risk(Z+a)=\rho(Z) + a$;
\item[A4] Positive homogeneity: if $\lambda\geq0$ and $Z\in \cs$, then $\risk(\lambda Z) = \lambda \risk(Z)$.
\end{description}
\end{definition} 
These axioms were originally conceived in \cite{artzner_delbaen_eber_heath_98} and ensure the ``rationality" of single-period risk assessments (we refer the reader to \cite{artzner_delbaen_eber_heath_98} for a detailed motivation of these axioms). One of the main properties for coherent risk metrics is a universal representation theorem for coherent risk metrics, which in the context of \emph{finite} probability spaces takes the following form:
\begin{theorem}[Representation Theorem for Finite Probability Spaces {\cite[page 265]{Shapiro_Dentcheva_Ruszczynski_09}}]\label{thrm:rep_finite}
Consider the probability space $\{\Omega, \fil, \probnoarg\}$ where $\Omega$ is finite, i.e., $\Omega=\{\omega_1, \ldots, \omega_L\}$, $\fil$ is the $\sigma$-algebra of all subsets (i.e., $\fil = 2^{\Omega}$), and $\probnoarg = (p(1), \ldots, p(L))$, with all probabilities positive.  Let $\mathcal B$ be the set of probability density functions:
\[{\small
\mathcal B:=\Bigl \{ \zeta\in \reals^L: \sum_{j=1}^L \, p(j)\zeta(j)=1, \zeta\geq 0 \Bigr\}.}
\]
The risk measure $\risk:\cs \rightarrow \reals$ is a coherent risk measure if and only if there exists a convex bounded and weakly* closed set $\csd \subset \mathcal B$ such that $
\risk(Z)=\max_{\zeta\in \csd} \mathbb E_{\zeta}[Z]$.
\end{theorem}


The result essentially states that any coherent risk measure is an expectation with respect to a worst-case density function $\zeta$, chosen adversarially from a suitable \emph{set} of test density functions  (referred to as  \emph{risk envelope}). 

%


\subsection{Dynamic, time-consistent  measures of risk}\label{subsection_time_cons}

This section provides a multi-period generalization of the concepts presented in Section \ref{subsec:static_risk} and follows closely the discussion in \cite{rus_09}. Consider a probability space $(\Omega, \fil, \mathbb P)$, a filtration $\fil_0\subset \fil_1\subset \fil_2 \cdots \subset \fil_N \subset \fil$, and an adapted sequence of real-valued random variables $Z_k$, $k\in \{0, \ldots,N\}$. We assume that $\fil_0 = \{\Omega, \emptyset\}$, i.e., $Z_0$ is deterministic. The variables $Z_k$ can be interpreted as stage-wise costs. For each $k\in\{0, \ldots, N\}$, denote with $\cs_k$ the space of random variables defined over the probability space $(\Omega, \fil_k, \mathbb P)$; also, let $\cs_{k, N}:=\cs_k \times \cdots \times \cs_N$. Given sequences $Z = \{Z_k,\ldots, Z_N\}\in \cs_{k, N}$ and $W=\{W_k,\ldots, W_N\}\in \cs_{k, N}$, we interpret $Z\leq W$ component-wise, i.e., $Z_j\leq W_j$ for all $j\in \{k,\ldots, N\}$.

The fundamental question in the theory of dynamic risk measures is the following: how do we evaluate the risk of the sequence $\{Z_k, \ldots, Z_N\}$ from the perspective of
stage $k$? The answer, within the modern theory of risk, relies on two key intuitive facts \cite{rus_09}. First, in dynamic settings  the specification of risk preferences should no longer entail constructing a single risk metric but rather a \emph{sequence} of risk metrics  $\{\risk_{k,N}\}_{k=0}^{N}$, each mapping a future stream of random costs into a risk metric/assessment at time $k$. This motivates the following definition.

\begin{definition}[Dynamic Risk Measure]
A dynamic risk measure is a sequence of mappings  $\risk_{k,N}:\cs_{k, N}\rightarrow\cs_k$, $k\in\{0, \ldots,N\}$, obeying the following monotonicity property:
\[
\risk_{k,N}(Z)\!\leq \!\risk_{k,N}(W)   \text{ for all } Z,W \!\in\!\cs_{k,N} \text{ such that } Z\leq W.
\]
\end{definition}
The above monotonicity property (analogous to axiom A2 in Definition \ref{def:crm}) is, arguably, a natural requirement for any meaningful dynamic risk measure.

Second, the sequence of metrics $\{ \risk_{k,N} \}_{k=0}^{N}$ should be constructed so that  the risk preference profile is \emph{consistent} over time \cite{Cheridito_09, shapiro_09, Iancu_11}. A widely accepted notion of time-consistency is as follows \cite{rus_09}: if a certain outcome is considered less risky in all states of the world at stage $k+1$, then it should also be considered less risky at stage $k$. 

The following example (adapted from \cite{Roorda:05}) shows how dynamic risk measures as defined above might indeed result in \emph{time-inconsistent}, and ultimately undesirable, behaviors.
\begin{example}\label{ex:paradox}
Consider the simple setting whereby there is a final cost $Z$ and one seeks to evaluate such cost from the perspective of earlier stages. Consider the three-stage scenario tree in Figure \ref{fig:sm}, with the elementary events $\Omega = \{UU, UD, DU, DD \}$, and the filtration $\fil_0=\{\emptyset, \Omega\}$, $\fil_1 = \Bigl \{\emptyset, \{U\}, \{D\}, \Omega \Bigr\}$, and $\fil_2 = 2^{\Omega}$. Consider the dynamic risk measure:
\[
\risk_{k,N}(Z) := \max_{q \in \mathcal \csd}\mathbb{E}_{q}[Z | \fil_k], \quad k=0,1,2
\] 
where $\csd$ contains two probability measures, one corresponding to $p =0.4$, and the other one to $p =0.6$
Assume that the random cost is $Z(UU) = Z(DD) = 0$, and $Z(UD) = Z(DU)=100$. Then, one has $\risk_1(Z)(\omega) = 60$ for all $\omega$, and $\risk_0(Z)(\omega) = 48$.  Therefore, $Z$ is deemed strictly riskier than a deterministic cost $W=50$ in \emph{all} states of nature at time $k=1$, but nonetheless $W$ is deemed riskier than $Z$ at time $k=0$ -- a paradox!
\end{example}
\vspace{-0.5truecm}
\begin{figure}[h]
\centering
{
  \includegraphics[width = 0.3\textwidth]{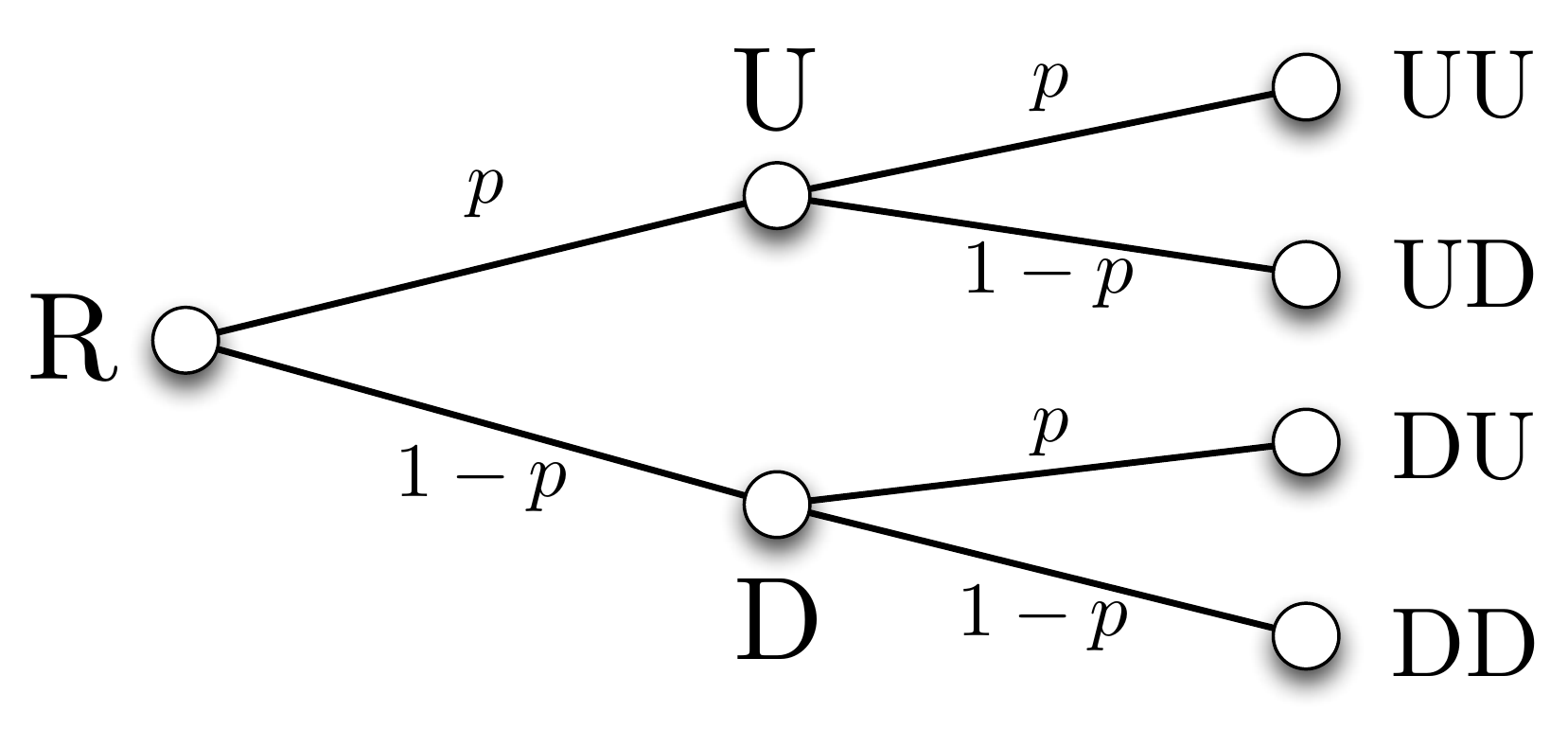}
}
      \caption{Scenario tree for example \ref{ex:paradox}.}
      \label{fig:sm}.
\end{figure}
\vspace{-0.5truecm}
It is important to note that there is nothing special about the selection of this example, similar paradoxical results could be obtained with other risk metrics. We refer the reader to \cite{rus_09, shapiro_09, Iancu_11} for further insights into the notion of time consistency and its practical relevance.
The issue then is what additional ``structural" properties are required for a dynamic risk measure to be time consistent. We first provide a rigorous version of the previous definition of time-consistency.
\begin{definition}[Time Consistency (\cite{rus_09})]
A dynamic risk measure $\{ \risk_{k,N}\}_{k=0}^N$ is time-consistent if, for all $0\leq l<k\leq N$ and all sequences $Z, W \in \cs_{l,N}$, the conditions
\begin{equation}
\begin{split}
&Z_i = W_i,\,\, i = l,\ldots,k-1, \text{ and }\\
&\risk_{k,N}(Z_k, \ldots,Z_N)\leq \risk_{k,N}(W_k, \ldots,W_N),
\end{split}
\end{equation}
imply that
\[
 \risk_{l,N}(Z_l, \ldots,Z_N)\leq \risk_{l,N}(W_l, \ldots,W_N).
\]
\end{definition}


As we will see in Theorem \ref{thrm:tcc}, the notion of time-consistent risk measures is tightly linked to the notion of coherent risk measures, whose generalization to the multi-period setting is given below:
\begin{definition}[Coherent One-step Conditional Risk Measures (\cite{rus_09})]
A coherent one-step conditional risk measure is a mapping $\risk_k:\cs_{k+1}\rightarrow \cs_k$, $k\in\{0,\ldots,N-1\}$, with the following four properties:
\begin{itemize}
\item Convexity: $\risk_k(\lambda Z + (1-\lambda)W)\leq \lambda\risk_k(Z) + (1-\lambda)\risk_k(W)$, $\forall \lambda\in[0,1]$ and $Z,W \in\cs_{k+1}$;
\item Monotonicity:  if $Z\leq W$ then $\risk_k(Z)\leq\risk_k(W)$, $\forall Z,W \in\cs_{k+1}$;
\item Translation invariance:  $\risk_k(Z+W)=Z + \risk_k(W)$, $\forall Z\in\cs_k$ and $W \in \cs_{k+1}$;
\item Positive homogeneity: $\risk_k(\lambda Z) = \lambda \risk_k(Z)$, $\forall Z \in \cs_{k+1}$ and $\lambda\geq 0$.
\end{itemize}
\end{definition} 

%
We are now in a position to state the main result of this section.
\begin{theorem}[Dynamic, Time-consistent Risk Measures (\cite{rus_09})]\label{thrm:tcc}
Consider, for each $k\in\{0,\ldots,N\}$, the mappings $\risk_{k,N}:\cs_{k, N}\rightarrow\cs_k$ defined as
\begin{equation}\label{eq:tcrisk}
\begin{split}
\risk_{k,N} = Z_k& + \risk_k(Z_{k+1} + \risk_{k+1}(Z_{k+2}+\ldots+\\&
\risk_{N-2}(Z_{N-1}+\risk_{N-1}(Z_N))\ldots)),
\end{split}
\end{equation}
where the $\risk_k$'s are coherent one-step conditional risk measures. Then, the ensemble of such mappings is a dynamic, time-consistent  risk measure.
\end{theorem}

Remarkably, Theorem 1 in \cite{rus_09} shows (under weak assumptions) that the ``multi-stage composition" in equation \eqref{eq:tcrisk} is indeed \emph{necessary for time consistency}. Accordingly, in the remainder of this paper, we will focus on the \emph{dynamic, time-consistent risk measures} characterized in Theorem \ref{thrm:tcc}.

\section{Model Description}\label{sec:sys}
Consider the discrete time system:
\begin{equation}
x_{k+1}=A(w_k)x_k+B(w_k)u_k,\label{eqn_sys}
\end{equation}
where $k\in \naturals$ is the time index, $x_k\in\reals^{N_x}$ is the state, $u_k\in\reals^{N_u}$ is the (unconstrained) control input, and $w_k\in\mathcal{W}$ is the process disturbance. We assume that the initial condition $x_0$ is deterministic. We assume that $\mathcal W$ is a finite set of cardinality $L$, i.e., $\mathcal W = \{w^{[1]}, \ldots, w^{[L]}\}$. For each stage $k$  and state-control pair $(x_k, u_k)$, the process disturbance $w_k$ is drawn from set $\mathcal W$ according to the  probability mass function 
\[
p=[p(1),\, p(2),\ldots,\, p(L)]^T,
\]
where $p(j)=\mathbb{P}(w_k=w^{[j]})$, $j\in\{1,\ldots,L\}$. Without loss of generality, we assume that $p(j)>0$ for all $j$. Note that the probability mass function for the process disturbance is time-invariant, and that the process disturbance is \emph{independent} of the process history and of the state-control pair $(x_k, u_k)$. Under these assumptions, the stochastic process $\{x_k\}$ is clearly a Markov process.

By enumerating all $L$ realizations of the process disturbance $w_k$, system \eqref{eqn_sys} can be rewritten as:
\[{\small
x_{k+1}=\left\{\begin{array}{cc}
A_{1}x_k+B_{1}u_k&\text{if $w_k=w^{[1]}$},\\
\vdots&\vdots\\
A_{L}x_k+B_{L}u_k&\text{if $w_k=w^{[L]}$},
\end{array}\right.}
\]
where $A_{j}:=A(w^{[j]})$ and $B_{j} := B(w^{[j]})$, $j\in \{1, \ldots, L\}$.

The results presented in this paper can be immediately extended to the time-varying case (i.e., where  the probability mass function for the process disturbance is time-varying). To simplify notation, however, we prefer to focus this paper on the time-invariant case.
\section{Markov Polytopic Risk Measures}\label{sec:risk}
In this section we \emph{refine} the notion of dynamic time-consistent risk metrics (as defined in Theorem \ref{thrm:tcc}) in two ways: (1) we add a polytopic structure to the dual representation of coherent risk metrics, and (2) we add a Markovian structure. This will lead to the definition of Markov dynamic polytopic risk metrics, which enjoy favorable computational properties and, at the same time,  maintain most of the generality of  dynamic time-consistent risk metrics.

\subsection{Polytopic risk measures}
According to the discussion in Section \ref{sec:sys}, the probability space for the process disturbance has a finite number of elements. Accordingly, consider Theorem \ref{thrm:rep_finite}; by definition of expectation, one has $\mathbb E_{\zeta}[Z]  =\sum_{j=1}^L\, Z(j) p(j) \zeta(j)$. In our framework (inspired by \cite{Eichhorn_05}), we consider coherent risk measures where the risk envelope $\csd$ is a \emph{polytope}, i.e., there exist matrices $S^I$, $S^E$ and vectors $T^I$, $T^E$ of appropriate dimensions such that 
\begin{equation*}\label{polytope_set_dual}
\upol=\left\{\zeta \in \mathcal B  \mid S^I \, \zeta \leq T^I,\,\, S^E \zeta= T^E \right\}.
\end{equation*}
We will refer to coherent risk measures representable with a polytopic risk envelope as \emph{polytopic risk measures}. Consider the bijective map $q(j):=p(j) \zeta(j)$ (recall that, in our model, $p(j)>0$). Then, by applying such map, one can easily rewrite a polytopic  risk measure as 
\[
\risk(Z)=\max_{q\in \upol} \mathbb E_{q}[Z],
\]
where $q$ is a \emph{probability mass function} belonging to a polytopic subset of the standard simplex, i.e.:
\begin{equation}\label{eq:rep_fin}
\upol= \Bigl \{ q\in \Delta^L \mid  S^I  q \leq T^I,\,\, S^E q= T^E  \Bigr  \},
\end{equation}
where $\Delta^L:=\bigl \{ q\in \reals^L: \sum_{j=1}^L \, q(j)=1, q \geq 0 \bigr\}$. Accordingly, one has $E_{q}[Z] = \sum_{j=1}^L \, Z(j) q(j)$ (note that, with a slight abuse of notation, we are using the same symbols as before for $\upol$, $S^I$, and $S^E$). 

The class of polytopic risk measures is large: we give below some examples (also note that any comonotonic risk measure is a polytopic  risk measure \cite{Iancu_11}).

\begin{example}(Examples of Polytopic Risk Measures)\label{example_expectation}
The expected value of a random variable $Z$  can be represented according to equation \eqref{eq:rep_fin} with the polytopic risk envelope 
$\upol =\bigl \{q\in \Delta^L  \mid q(j) = p(j) \quad \text{for all}\quad  j\in \{1,\ldots, L\}  \bigr\}$.


A second example is represented by the average upper semi-deviation risk metric, defined as
\[
\risk_{\mathrm{AUS}}(Z):=\expectation{Z}+c\, \mathbb{E}\bigl[(Z-\expectation{Z})^+\bigr],
\]
where $0\leq c\leq 1$. This metric can be represented according to equation \eqref{eq:rep_fin} with polytopic risk envelope (\cite{Ogryczak_Ruszczynski_99, Shapiro_Dentcheva_Ruszczynski_09}):
$\upol =\bigl\{q\in \Delta^L\mid q(j)=p(j)  \bigl(1+h(j)-\sum_{j=1}^L h(j) p(j) \bigr),
 0\leq h(j)\leq c,\,\, j\in\{1,\ldots,L\} \bigr\}$.

A third example is represented by the worst case risk, defined as
\[
\text{WCR}(Z):=\max\bigl\{Z(j):\,\, j\in\{1,\ldots,L\}\bigr\}.
\]
Such risk metric can be trivially represented according to equation \eqref{eq:rep_fin} with polytopic risk envelope $\upol = \Delta^L$.
\end{example}

Other important examples include the Conditional Value-at-Risk \cite{Rockafellar_Uryasev_00}, mean absolute semi-deviation \cite{Ogryczak_Ruszczynski_99}, the spectral risk measures  \cite{Bertsimas_Brown_09, Iancu_11},  the  optimized certainty equivalent and expected utility  \cite{Ben-Tal_07, Shapiro_Dentcheva_Ruszczynski_09, Eichhorn_05}, and the distributionally-robust risk  \cite{Bernardini_Bemporad_12}.  The key point is that the notion of polytopic risk metric  \emph{covers a full gamut of risk assessments}, ranging from risk-neutral to worst case. 
 

\subsection{Markov dynamic polytopic risk metrics}

Note  that in the definition of dynamic, time-consistent risk measures, since at stage $k$ the value of $\risk_k$ is $\fil_k$-measurable, the evaluation of risk can depend on the \emph{whole} past, see \cite[Section IV]{rus_09}. For example, the weight $c$ in the definition of the average upper mean semi-deviation risk metric  can be an $\fil_k$-measurable random variable (see \cite[Example 2]{rus_09}). This generality, which appears of little practical value in many cases, leads to optimization problems that are intractable. This motivates us to add a \emph{Markovian structure} to dynamic, time-consistent risk measures (similarly as in \cite{rus_09}). We start by introducing the notion of Markov polytopic risk measure (similar to \cite[Definition 6]{rus_09}).

\begin{definition}[Markov Polytopic Risk Measures]\label{def:Markov}
Consider the Markov process $\{x_k\}$ that evolves according to equation \eqref{eqn_sys}.  A coherent one-step conditional risk measure $\risk_k(\cdot)$ is a Markov polytopic risk measure with respect to $\{x_k\}$ if it can be written as
\[
\begin{split}
\risk_k(Z(x_{k+1}))=\max_{q\in\upol_{k}(x_k,p)}\mathbb{E}_{q}[Z(x_{k+1})]
\end{split}
\]
where
\[\begin{split}
 \upol_{k}(x_k, p)=
 \left\{q\in \Delta^L\mid 
 \begin{array}{l}
 S^I_{k}(x_k, p)q\leq T^I_{k}(x_k, p), \\
 S^E_{k}(x_k, p)q= T^E_{k}(x_k, p)
 \end{array}\!\!\right\}.
\end{split} \]
\end{definition}
In other words, a Markov polytopic risk measure is a coherent one-step conditional risk measure where the
evaluation of risk is not allowed to depend on the whole past (for example, the weight $c$  in the definition of the average upper mean semi-deviation risk metric  can depend on the past only through $x_k$), and the risk envelope is a polytope. Correspondingly, we define a Markov dynamic polytopic risk metric as follows.

\begin{definition}[Markov Dynamic Polytopic Risk Measures]\label{def:dyn_Mar}
Consider the Markov process $\{x_k\}$ that evolves according to equation \eqref{eqn_sys}. A Markov dynamic polytopic risk measure is a set of mappings $\risk_{k,N}:\cs_{k, N}\rightarrow\cs_k$ defined as
\begin{equation*}
\begin{split}
\risk_{k,N} = &Z(x_k) + \risk_k(Z(x_{k+1}) +\ldots+\\&
\qquad \risk_{N-2}(Z(x_{N-1})+\risk_{N-1}(Z(x_N)))\ldots)),
\end{split}
\end{equation*}
for $k\in\{0,\ldots,N\}$, where the single-period risk measures $\risk_k$ are Markov polytopic risk measures.
\end{definition}
Clearly, a Markov dynamic polytopic risk metric is time consistent. Definition \ref{def:dyn_Mar} can be extended to the case where the probability distribution for the disturbance depends on the current state and control action.  We avoid this generalization to keep the exposition simple and consistent with model \eqref{eqn_sys}.

\section{Problem Formulation}\label{sec:IHC}
In light of Sections \ref{sec:sys} and \ref{sec:risk}, we are now in a position to state the risk-averse MPC problem we wish to solve in this paper.  Our problem formulation relies on Markov dynamic polytopic risk metrics that satisfy the following stationarity assumption.
\begin{assumption}(Time-Invariance of Risk)\label{assume_tractable}
The polytopic risk envelopes $\upol_k$ do not depend explicitly on time and are independent of the state $x_k$, i.e. $\upol_k(x_k, p) = \upol(p)$  for all $k$. 
\end{assumption}
This assumption is crucial for the well-posedness of our formulation and  in order to devise  a tractable MPC algorithm that relies on linear matrix inequalities.  
We next  introduce a notion of stability tailored to our risk-averse context.
\begin{definition}[Uniform Global Risk-Sensitive Exponential Stabilty]\label{stoch_stab_defn_exp}
System \eqref{eqn_sys} is said to be Uniformly Globally Risk-Sensitive Exponentially Stable (UGRSES) if there exist constants $c\geq0$ and $\lambda\in[0,1)$ such that for all initial conditions $x_0\in\reals^{N_x}$,
\begin{equation}
\risk_{0,k}(0,\ldots,0,x_k^Tx_k)\leq c\, \lambda^k \, x_0^Tx_0,\quad \text{for all } k\in\naturals,\label{stab_ineq}
\end{equation}
where $\{\risk_{0,k}\}$ is a  Markov dynamic  polytopic risk measure satisfying Assumption \ref{assume_tractable}.
\end{definition}
One can easily show that UGRSES is a \emph{more restrictive} stability condition than mean-square stability (considered, for example, in \cite{Bernardini_Bemporad_12}).

Consider the MDP described in Section  \ref{sec:sys} and let $\Pi$ be the set of all stationary feedback control policies, i.e., $\Pi := \Bigl \{ \pi: \reals^{N_x} \rightarrow \reals^{N_u}\}$. Consider the quadratic cost function $c:\reals^{N_x}\times \reals^{N_u}\rightarrow \reals_{\geq 0}$ defined as 
\[
c(x,u):=x^T\,Q\,x+u^T\,R\,u,
\]
where $Q=Q^T\succ 0$ and $R=R^T\succ 0$ are given state and control penalties.
%
The problem we wish to address is as follows.
\begin{quote} {\bf Optimization Problem $\mathcal{OPT}$} --- Given an initial state $x_0\in \reals^{N_x}$, solve
\begin{alignat*}{2}
\inf_{\pi\in\Pi} & & \quad&\limsup_{k\rightarrow\infty}J_{0,k}(x_0,\pi)\\
\text{such that} & &\quad & x_{k+1}=A(w_k)x_{k}+B(w_k)\pi(x_k)\\
& &\quad&\text{System is UGRSES}\end{alignat*}
where
\[
J_{0,k}(x_0,\pi)=\risk_{0,k}\Bigl(c(x_0,\pi(x_0)),\ldots,c(x_{k},\pi(x_{k})) \Bigr),
\]
and $\{\risk_{0,k}\}$ is a Markov dynamic polytopic risk measure satisfying Assumption \ref{assume_tractable}.
\end{quote}
We will denote the optimal cost function as $J^{\ast}_\infty(x_0)$. Note that the risk measure in the definition of  UGRSES is assumed to be identical to the risk measure used to evaluate the  cost of a policy.  Also, note that, by Assumption \ref{assume_tractable}, the single-period risk metrics are time-invariant, hence one can write
\begin{equation}\label{eq:mar_pol_inv}
\begin{split}
&\risk_{0,k}\Bigl(c(x_0,\pi(x_0)),\ldots,c(x_{k},\pi(x_{k})) \Bigr) =c(x_0, \pi(x_0)) \\&\quad + \risk(c(x_1, \pi(x_1)) +\ldots+\risk(c(x_k, \pi(x_k)))\ldots),
\end{split}
\end{equation}
where $\risk$ is a  given Markov polytopic risk metric that models the ``amount" of risk aversion.  This paper addresses problem $\mathcal{OPT}$ along two main dimensions:
\begin{enumerate}
\item Find sufficient conditions for \emph{risk-sensitive} stability (i.e., for UGRSES).
\item Design a model predictive control algorithm to efficiently compute a suboptimal state-feedback control policy. 
\end{enumerate}

\section{Risk-Sensitive Stability}\label{sec:stab}
In this section we provide a sufficient condition for system \eqref{eqn_sys} to be UGRSES, under the assumptions of Section \ref{sec:IHC}. This condition relies on Lyapunov techniques and is inspired by \cite{Bernardini_Bemporad_12} (Lemma \ref{lyap_stab} indeed reduces to Lemma 1 in \cite{Bernardini_Bemporad_12} when the risk measure is simply an expectation).
\begin{lemma}\label{lyap_stab}
Consider a policy $\pi \in \Pi$ and the corresponding closed-loop dynamics for system \eqref{eqn_sys}, denoted by $x_{k+1}=f(x_k,w_k)$. The closed-loop system is UGRSES if there exists a function $V(x): \reals^{N_x}\rightarrow\reals$ and scalars $b_1,b_2,b_3>0$ such that for all $x\in\reals^{N_x}$ the following conditions hold:
\begin{equation}
\begin{split}
&b_1\, \|x\|^2\leq V(x)\leq b_2\|x\|^2,\\
& \risk(V(f(x, w)))-V(x)\leq -b_3\|x\|^2.\label{eqn_RSES}
\end{split}
\end{equation}
\end{lemma}
\begin{proof}
From the time consistency, monotonicity, translational invariance, and positive homogeneity of Markov dynamic polytopic risk measures, condition \eqref{eqn_RSES} implies
$b_1\,\,  \risk_{0,k+1}(0,\ldots,0,\|x_{k+1}\|^2)
 \leq\risk_{0,k+1}(0,\ldots,0, V(x_{k+1}))
 =\risk_{0,k}(0,\ldots,0,V(x_k)+\risk( V(x_{k+1})-V(x_k)))
 \leq\risk_{0,k}(0,\ldots,0,V(x_k)-b_3\|x_k\|^2)
 \leq\risk_{0,k}(0,\ldots,0,(b_2-b_3)\|x_k\|^2)$. Also, since $\risk_{0,k+1}$ is monotonic, one has $\risk_{0,k+1}(0,\ldots,0,\|x_{k+1}\|^2)\geq 0$, which implies $b_2\geq b_3$ and in turn $(1-b_3/b_2)\in[0,1)$. 
 
Since $V(x_k)/b_2\leq \|x_k\|^2$, and using the previous inequalities, one can write:
\begin{equation*}\small
\begin{split}
\risk_{0,k+1}(0,\ldots,0, V(x_{k+1})) \leq& \risk_{0,k}(0,\ldots,0,V(x_k)-b_3\|x_k\|^2)\\
\leq& \left(1\!-\!\frac{b_3}{b_2}\right)\risk_{0,k}\left(0,\ldots,0,V(x_k)\right).
\end{split}
\end{equation*}
Repeating this bounding process, one obtains:
\begin{equation*}\small
\begin{split}
&\risk_{0,k+1}(0,\ldots,0,V(x_{k+1}))\leq\left(1-\frac{b_3}{b_2}\right)^{k}\risk_{0,1}\left(V(x_1)\right)\\
&\quad \, \, =\left(1-\frac{b_3}{b_2}\right)^{k} \risk \left(V(x_1)\right)\leq \left(1-\frac{b_3}{b_2}\right)^{k} \left(V(x_0) - b_3\|x_0\|^2 \right)\\
&\quad\, \,  \leq\, b_2\left(1-\frac{b_3}{b_2}\right)^{k+1} \, \|x_0\|^2.
\end{split}
\end{equation*}
Again, by monotonicity, the above result implies
\[
\begin{split}
\risk_{0,k+1}(0,\ldots,0,x_{k+1}^Tx_{k+1})\leq&\frac{b_2}{ b_1}\left(1-\frac{b_3}{b_2}\right)^{k+1}x_0^Tx_0.
\end{split}
\]
By setting $c=b_2/ b_1$ and $\lambda=(1-b_3/b_2)\in[0,1)$,  the claim is proven.
\end{proof}

\section{Model Predictive Control Problem}\label{sec:MPC}
In this section we set up the receding horizon version of problem $\mathcal{OPT}$. This will lead to a  model predictive control algorithm for the (suboptimal) solution of problem $\mathcal{OPT}$. Consider the following receding-horizon cost function for $N\geq 1$:

\begin{equation}\label{cost_MPC}\small
\begin{split}
&J(x_{k|k},\pi_{k|k},\ldots,\pi_{k+N-1|k},P) :=\risk_{k,k+N}\big(c(x_{k|k},\pi_{k|k}(x_{k|k})),\\
&\ldots,c(x_{k+N-1|k},\pi_{k+N-1|k}(x_{k+N-1|k}), x_{k+N}^T Px_{k+N}\big),
\end{split}
\end{equation}
where $x_{h|k}$ is the state at time $h$ predicted at stage $k$ (a \emph{discrete} random variable), $\pi_{h|k}$ is the control \emph{policy} to be applied at time $h$ as determined at stage $k$ (i.e., $\pi_{h|k}:\reals^{N_x}\rightarrow \reals^{N_u}$), and $P=P^T\succ 0$ is a terminal weight matrix. We are now in a position  to  state the model predictive control problem.
\begin{quote} {\bf Optimization problem $\mathcal{MPC}$} --- Given an initial state $x_{k|k}\in  \reals^{N_x}$ and a prediction horizon $N\geq 1$, solve
\begin{alignat*}{2}
\min_{\pi_{k|k},\ldots,\pi_{k+N-1|k}}& & \quad&\!J\left(x_{k|k},\pi_{k|k},\ldots,\pi_{k+N-1|k},P\right) \\
\text{such that} & &\!\quad & x_{k+h+1|k}=A(w_{k+h})x_{k+h|k}+\\
 & & \quad & \qquad   \qquad B(w_{k+h})\pi_{k+h|k}(x_{k+h|k})
\end{alignat*}
for $h\in \{0, \ldots, N-1\}$.
\end{quote}
Note that a Markov policy is guaranteed to be optimal for problem $\mathcal{MPC}$ (see \cite[Theorem 2]{rus_09}). The optimal cost function for problem $\mathcal{MPC}$ is denoted by $J^{*}_k(x_{k|k})$, and the set of minimizers is denoted by $\{\pi^\ast_{k|k}, \ldots, \pi^\ast_{k+N-1|k}\}$. For each state $x_k$, we set $x_{k|k} = x_k$ and the (time-invariant) model predictive control law is then defined as 
\begin{equation}\label{MPC_law}
\begin{split}
 \pi^{MPC}(x_k)=&\pi^\ast_{k|k}(x_{k|k}).
\end{split}
\end{equation}
Note that the model predictive control problem $\mathcal{MPC}$ involves an optimization over \emph{time-varying closed-loop policies}, as opposed to the classical deterministic case where the optimization is over open-loop sequences. A similar approach is taken in \cite{Primbs_Sung_09, Bernardini_Bemporad_12}. We will show in Section \ref{sec:alg} how to solve problem $\mathcal{MPC}$ efficiently.

The following theorem shows that the model predictive control law \eqref{MPC_law}, with a proper choice of the terminal weight $P$, is risk-sensitive stabilizing, i.e., the closed-loop system \eqref{eqn_sys} is UGRSES.

\begin{theorem} (Stochastic Stability for Model Predictive Control Law)\label{stoch_stab_MPC}
Consider the model predictive control law in equation \eqref{MPC_law} and the corresponding closed-loop dynamics for system \eqref{eqn_sys} with initial condition $x_{0}\in \reals^{N_x}$. Suppose  that $P=P^T\succ 0$, and there exists a matrix $F$ such that:
\begin{equation}\label{term_ineq}
\sum_{j=1}^L q_{l}(j)\, (A_{j}+B_{j}F)^TP(A_j+B_{j}F)-P+(F^TRF+Q)\prec 0,
\end{equation}
for all $l\in\{1,\ldots,\mathrm{cardinality}(\upolv(p))\}$, where $\upolv(p)$ is the set of vertices of polytope $\upol(p)$, $q_l$ is the $l$th element in set $\upolv(p)$, and $q_l(j)$ denotes the $j$th component of vector $q_l$, $j\in\{1\, \ldots, L\}$.  Then, the closed loop system \eqref{eqn_sys} is UGRSES. 
\end{theorem}
\begin{proof}
The strategy of this proof is to show that $J^*_k$ is a valid Lyapunov function in the sense of Lemma \ref{lyap_stab}. Specifically, we want to show that $J^*_k$ satisfies the two inequalities in equation \eqref{eqn_RSES}; the claim then follows by simply noting that, in our time-invariant setup,  $J^*_k$ does not depend on $k$.

We start by focusing on the bottom inequality in equation \eqref{eqn_RSES}. Consider a time $k$ and an initial condition $x_{k|k}\in \reals^{N_x}$ for problem $\mathcal{MPC}$. The sequence of optimal control policies  is given by $\{\pi^*_{k+h|k} \}_{h=0}^{N-1}$. Let us define a sequence of control policies from time $k+1$ to $N$ according to
\begin{equation}\label{ctrl_seq_MI}
\pi_{k+h|k+1}(x_{k+h|k}):=\!\left\{\!\begin{array}{ll}
 \pi^*_{k+h|k}(x_{k+h|k}) &  \text{if $h\in[1,N-1]$},\\
 F\, x_{k+N|k} & \text{if $h=N$}.\\
 \end{array}\right.
 \end{equation}
 This sequence of control policies is essentially the concatenation of the sequence  $\{\pi^*_{k+h|k} \}_{h=1}^{N-1}$ with a linear feedback control law for stage $N$ (the reason why we refer to this policy with the subscript  ``$k+h|k+1$" is that we will use this policy as a feasible policy for problem $\mathcal{MPC}$ starting at stage $k+1$). 
 
Consider the $\mathcal{MPC}$ problem at stage $k+1$ with initial condition given by $x_{k+1|k+1}=A(w_k)x_{k|k}+B(w_k)\pi_{k|k}^\ast(x_{k|k})$, and denote with $\overline{J}_{k+1}(x_{k+1|k+1})$ the cost of the objective function assuming that the sequence of control policies is given by $\{\pi_{k+h|k+1}\}_{h=1}^N$. Note that  $x_{k+1|k+1}$ (and therefore $\overline{J}_{k+1}(x_{k+1|k+1})$) is a random variable with $L$ possible realizations. Define:
\begin{equation*}\small
\begin{split}
&Z_{k+N}:=-x_{k+N|k}^TP x_{k+N|k}+x_{k+N|k}^TQ x_{k+N|k}\\
&\qquad\qquad \qquad+(F\, x_{k+N|k})^T R\, F\, x_{k+N|k},\\
&Z_{k+N+1}:=\Bigl( (A(w_{k+N|k}) + B(w_{k+N|k})F) x_{k+N|k} \Bigr)^T\cdot\\
&\qquad \qquad \,P \,  \Bigl( (A(w_{k+N|k}) + B(w_{k+N|k})F) x_{k+N|k} \Bigr).
\end{split}
\end{equation*}
By exploiting the dual representation of Markov polytopic risk metrics, one can easily show that equation \eqref{term_ineq} implies
\begin{equation}\label{eq:ineq_key}
Z_{k+N}+\risk(Z_{k+N+1})\leq 0.
\end{equation}
One can then write the following chain of inequalities:
\begin{equation}\label{eq:bottom}\small
\begin{split}
J^{*}_k&(x_{k|k})=x_{k|k}^TQ x_{k|k}\!+\!(\pi_{k|k}^\ast(x_{k|k}))^TR \pi_{k|k}^\ast(x_{k|k})+\\
& \risk\Biggl(\risk_{k+1, N}\Bigl(c(x_{k+1|k}, \pi^*_{k+1}(x_{k+1|k})), \ldots,  x_{k+N|k}^TQ x_{k+N|k}      +       \\
&(F x_{k+N|k})^T R F x_{k+N|k}\! +\!\risk(Z_{k+N+1}) \!-\! Z_{k+N} \!- \!\risk({Z_{k+N+1}}) \!\Bigr) \!\!\Biggr)\! \!\geq\\
&x_{k|k}^TQ x_{k|k}\!+\!(\pi_{k|k}^\ast(x_{k|k}))^TR \pi_{k|k}^\ast(x_{k|k})+\\
& \risk\Biggl(\risk_{k+1, N}\Bigl(c(x_{k+1|k}, \pi^*_{k+1}(x_{k+1|k})), \ldots,  x_{k+N|k}^TQ x_{k+N|k}      +       \\
&(F\, x_{k+N|k})^T R\, F\, x_{k+N|k} +\risk(Z_{k+N+1}) \Bigr) \Biggr)=\\
&x_{k|k}^TQ x_{k|k}\!+\!(\pi_{k|k}^\ast(x_{k|k}))^TR \pi_{k|k}^\ast(x_{k|k})\!+\!\risk\Bigl({\overline{J}_{k+1}(x_{k+1|k+1})}\Bigr)\!\geq\\
&x_{k|k}^TQ x_{k|k}\!+\!(\pi_{k|k}^\ast(x_{k|k}))^TR \pi_{k|k}^\ast(x_{k|k})+\risk\Bigl(J^*_{k+1}(x_{k+1|k+1})\Bigr),
\end{split}
\end{equation}
where the first equality follows from our definitions, the second inequality follows from equation \eqref{eq:ineq_key}  and the monotonicity property of Markov polytopic risk metrics (see also \cite[Page 242]{rus_09}), the third equality 
follows from the fact that the sequence of control policies $\{\pi_{k+h|k+1}\}_{h=1}^N$ is a feasible sequence for the $\mathcal{MPC}$ problem starting at stage $k+1$ with initial condition $x_{k+1|k+1}=A(w_k)x_{k|k}+B(w_k)\pi_{k|k}^\ast(x_{k|k})$, and the fourth inequality follows form the definition of $J^*_{k+1}$ and the monotonicity of Markov polytopic risk metrics.

We now turn our focus to the top inequality in equation \eqref{eqn_RSES}. One can easily bound $J^{*}_k(x_{k|k})$ from below according to:
\begin{equation}\label{eq:low}
J^{*}_k(x_{k|k})\geq x_{k|k}^T Q x_{k|k}\geq \lambda_{\min}(Q)\|x_{k|k}\|^2,
\end{equation}
where $\lambda_{\min}(Q)>0$ by assumption.  To bound $J^{*}_k(x_{k|k})$ from above, define:
\[
\begin{split}
& M_A:=\max_{r\in\{0,\ldots,N-1\}}\max_{j_0,\ldots,j_{r}\in\{1,\ldots,L\}}\|A_{j_{r}}\ldots A_{j_{1}}A_{j_0}\|_2.
\end{split}
\] 
One can write:
\begin{equation*}\small
\begin{split}
&J^{*}_k(x_{k|k})\leq c(x_{k|k},0)+\risk\Bigl( c(x_{k+1|k},0)+ \risk\Bigl( c(x_{k+2|k},0)\\
&\quad\qquad\qquad+\ldots+\risk\left(x_{k+N|k}^TP x_{k+N|k}\right)\ldots\Bigr)\Bigr)\\
\leq&\|Q \|_2\|x_{k|k}\|_2^2 +\risk\Bigl( \|Q \|_2\|x_{k+1|k}\|_2^2+\risk\Bigl(\|Q \|_2\|x_{k+2|k}\|_2^2+\\
&\qquad\ldots+\risk\left( \|P\|_2\|x_{k+N}\|^2_2\right)\ldots\Bigr)\Bigr).
\end{split}
\end{equation*}
Therefore, by using the translational invariance and monotonicity property of Markov polytopic risk metrics, one obtains the upper bound:
\begin{equation}\label{eq:up}
0\leq J^{*}_k(x_{k|k})\leq \left(N\, \|Q \|_2+ \|P\|_2\right )M_A\|x_{k|k}\|_2^2.
\end{equation}
Combining the results in equations \eqref{eq:bottom}, \eqref{eq:low}, \eqref{eq:up}, and given the time-invariance of our problem setup, one concludes that $J^{*}_k(x_{k|k})$ is a ``risk-sensitive" Lyapunov function for the closed-loop system \eqref{eqn_sys}, in the sense of Lemma \ref{lyap_stab}. This concludes the proof.
\end{proof}
The study of the conservatism of the LMI condition \eqref{term_ineq} is left for future work. We note here that out of 100 \emph{randomly} generated 5-state-2-action-3-scenario examples, 83 of them satisfied condition \eqref{term_ineq}.

\section{Solution Algorithms}\label{sec:alg}
In this section we discuss two solution approaches, the first one via dynamic programming, the second one via convex programming.

\subsection{Dynamic programming approach}

Problem $\mathcal{MPC}$ can be solved via dynamic programming, see \cite[Theorem 2]{rus_09}. However, one would first need to find a matrix $P$ that satisfies equation \eqref{term_ineq}. The next theorem presents a linear matrix inequality characterization of condition \eqref{term_ineq}.

%

\begin{theorem}\label{thm_stab_1}
Define $\overline{A}=\begin{bmatrix}
A_1^T&\ldots&A_{L}^T
\end{bmatrix}^T$, $\overline{B}=\begin{bmatrix}
B_{1}^T&\ldots&B_{L}^T
\end{bmatrix}^T$ and $\Sigma_{l}=\text{diag}(q_{l}(1)I,\ldots,q_{l}(L) I)\succeq 0$, for all $q_{l}\in\upolv(p)$. 
Consider the following set of  linear matrix inequalities with respect to decision variables $Y$, $G$, $\overline{Q}=\overline{Q}^T\succ 0$:
\begin{equation} \label{ineq_stab_1}{\small
\begin{split}
&\begin{bmatrix}
I_{L\times L}\otimes\overline{Q}&0&0&-\Sigma^{\frac{1}{2}}_{l}(\overline{A}G+\overline{B}Y)\\
\ast&R^{-1}&0&-Y\\
\ast&\ast&I&-Q^{\frac{1}{2}}G\\
\ast&\ast&\ast&-\overline{Q}+G+G^T\\
\end{bmatrix}\succ 0,\\
\end{split}}
\end{equation}
where  $l\in\{1,\ldots,\mathrm{cardinality}(\upolv(p))\}$. The set of linear matrix inequalities in \eqref{ineq_stab_1} is equivalent to the condition in \eqref{term_ineq} by setting $F=YG^{-1}$ and $P=\overline Q^{-1}$.
\end{theorem}
\begin{proof}
The theorem can be proven by using the Projection Lemma (see \cite[Chapter 2]{Skelton_Iwasaki_Grigoriadis_98}). The details are omitted in the interest of brevity.
\end{proof}

Hence, a solution approach for $\mathcal{MPC}$ is to first solve the linear matrix inequality in Theorem \ref{thm_stab_1} and then, if a solution for $P$ is found, apply dynamic programming (after state and action \emph{discretization}, see, e.g., \cite{Chow_Pavone_13_2,chow1991optimal}). Note that the discretization process might yield a large-scale dynamic programming problem, which motivates the convex programing approach presented next.

\subsection{Convex programming approach}

Consider the following parameterization of \emph{history-dependent} control policies. Let  $j_0,\ldots,j_{h}\in\{1,\ldots,L\}$ be the realized indices for the disturbances in the first $h+1$ stages of the $\mathcal{MPC}$ problem, where $h\in \{0, \ldots, N-1\}$; for $h\geq 1$, we will refer to the control to be exerted at stage $h$ as $\overline{U}_h(j_0,\ldots,j_{h-1})$. Similarly, we will refer to the state at stage $h$ as $\overline{X}_h(j_0,\ldots,j_{h-1})$. The quantities $\overline{X}_h(j_0,\ldots,j_{h-1})$ and $\overline{U}_h(j_0,\ldots,j_{h-1})$ enable us to keep track of the (exponential) growth of the scenario tree. In terms of this new notation, the system dynamics \eqref{eqn_sys} can be rewritten according to:
\begin{equation}\label{stoch_trans_ALG}
\begin{split}
&\overline{X}_0:=x_{k|k},\,\, \overline{U}_0 \in \reals^{N_u},\\
&\overline{X}_1(j_0)=A_{j_0}\overline{X}_{0}+ B_{j_0}\overline{U}_{0},\,\, \text{for $h=1$},\\
&\overline{X}_h(j_0,\ldots,j_{h-1})=A_{j_{h-1}}\overline{X}_{h-1}(j_0,\ldots,j_{h-2})+\\
&\qquad \qquad B_{j_{h-1}}\overline{U}_{h-1}(j_0,\ldots,j_{h-2}),\,\,\text{for $h\geq 2$}.
 \end{split}
\end{equation}

Indeed, the $\mathcal{MPC}$ problem is defined as an optimization problem over \emph{Markov} control policies. However, in the convex programming approach, we re-define the  $\mathcal{MPC}$ problem as an optimization problem over \emph{history-dependent} policies. One can show (with a virtually identical proof) that the stability Theorem \ref{thm_stab_1} still holds when history-dependent policies are considered. Furthermore, since Markov policies are optimal in our setup (see \cite[Theorem 2]{rus_09}), the value of the optimal cost stays the same. The key advantage of history-dependent policies is that their additional flexibility leads to a convex optimization problem for the determination of the model predictive control law. Specifically, the model predictive control law can be determined according to the following algorithm.

\begin{quote}{\bf Solution for $\mathcal{MPC}$} --- Given an initial state $x_{k|k} = x_k$ and a prediction horizon $N\geq 1$, solve:
\[
\min_{\footnotesize \begin{array}{c}
\gamma_1,W, G_{j_N}, Y_{j_N}, \overline Q, {\gamma}_{2}(j_0,\ldots,j_{N-1})\\
\overline{U}_0,\,\overline{U}_{h}(j_0,\ldots,j_{h-1}),h\in\{1,\ldots,N\}\\
\overline{X}_{h}(j_0,\ldots,j_{h-1}),h\in\{1,\ldots,N\}\\
j_0,\ldots,j_{N-1}\in\{1,\ldots,L\}
\end{array}}\quad \gamma_1
\]
subject to 
\begin{itemize}
\item The linear matrix inequality in equation (\ref{ineq_stab_1}).
\item The linear matrix inequality
\begin{equation*}
\begin{bmatrix}
 {\gamma}_{2}(j_0,\ldots,j_{N-1}) I&\overline{X}_{N}(j_0,\ldots,j_{N-1})^T\\
 \ast&\overline{Q}
\end{bmatrix}\succeq 0.
\end{equation*}
\item The system dynamics in equation \eqref{stoch_trans_ALG}.
\item The objective epigraph constraint:
\[
\begin{split}
\risk_{k,k+N}(c_0(\overline{X}_{0},\overline{U}_{0}),&\ldots,c_{N-1}(\overline{X}_{N-1},\overline{U}_{N-1}),\\
&\gamma_2(j_0,\ldots,j_{N-1}))\leq \gamma_1.
\end{split}
\]
\end{itemize}
Then, set $\pi^{MPC}(x_k) = \overline{U}_{0}$.
\end{quote}

The proof of the correctness of this algorithm is omitted due to lack of space and will be presented in the journal version of this paper. This algorithm is clearly suitable only for ``moderate" values of $N$, given the combinatorial explosion of the scenario tree. Indeed, one could devise a \emph{hybrid} algorithm where the above computation is split into two steps: (1) the terminal cost function is computed offline, and (2) the MPC control law is computed online. We tested this approach and found that the computation time is very short for reasonably large problems. The details are deferred to a future publication.

\section{Numerical Experiments}\label{sec:example}
In this section we present numerical experiments run on a 2.3 GHz Intel Core i5, MacBook Pro laptop, using the MATLAB YALMIP Toolbox, version 2.6.3 \cite{YALMIP_08}. Consider the stochastic system: $x_{k+1}=A(w_k)x_k+B(w_k)u_k$, where $w_k\in\{1,2,3\}$ and 
\[{\small
\begin{split}
&A_1=\begin{bmatrix}
2  &   0.5\\
     -0.5 &    2
     \end{bmatrix}\!, A_2=\begin{bmatrix}
     0.01 &0.1\\
     0.05 &0.01    \end{bmatrix}\!, A_3=\begin{bmatrix}
1.5 &-0.3\\
0.2 &1.5
\end{bmatrix}\!,\\
&B_1=\begin{bmatrix}
3&0.1\\
0.1&3
    \end{bmatrix},\,\,B_2=\begin{bmatrix}
1&0.5\\
0.5&1\\
   \end{bmatrix},\,\,B_3=\begin{bmatrix}
2&0.3\\
0.3&2   \end{bmatrix}.
   \end{split}}
\]
The probability mass function for the process disturbance is uniformly distributed, i.e., $\mathbb{P}(w_k=i)=\frac{1}{3}$, for $i\in\{1,2,3\}$. In this example, we want to explore the risk aversion capability of the risk-averse MPC  algorithm presented in Section \ref{sec:MPC} (the solution relies on the convex programming approach). We consider as risk-aversion metric the mean upper semi-deviation metric (see Section \ref{sec:prelim}), where $c$ ranges in the set $\{0,0.25,0.5,0.75,1\}$. The initial condition is $x_{0}(1)=x_{0}(2)=1$ and the number of lookahead steps is $3$. We set $Q=1\times I_{2\times 2}$ and $R=10^{-4}  I_{2\times 2}$. We perform 100 Monte Carlo simulations for each value of $c$. When $c=0$, the problem is reduced to a risk-neutral minimization problem. On the other hand, one enforces maximum emphasis on regulating semi-deviation (dispersion) by setting $c=1$. Table \ref{table_aver} summarizes our results (computation times are given in seconds). When $c=0$ (risk neutral formulation), the average cost is the lowest (with respect to the different choices for $c$), but the dispersion is the largest. Conversely, when $c=1$, the dispersion is the lowest, but the average cost is the largest. 
\begin{center}
\begin{table}
     \caption{Statistics for Risk-Averse MPC.} \label{table_aver}
       \begin{tabular}{ l l l l l l l}
         \hline
     Level &Sample  &Sample & Sample& Mean Time\\
    of risk &Mean &Dispersion & Standard Deviation &Per Itr. (Sec)\\
     \hline
      c=0 &2.9998 &0.2889&0.4245&4.8861 \\
      c=0.25 &3.3012&0.2643&0.3520&5.2003\\
      c=0.5 &3.4178&0.2004&0.2977&4.4007\\
      c=0.75 &3.5898&0.1601&0.2231&4.4577\\
      c=1&3.6072&0.0903&0.1335&4.6498\\
     \hline
  \end{tabular}
    \end{table}
\end{center}

\section{Conclusion and Future Work}\label{sec:conclusion}
In this paper we presented a framework for risk-averse MPC. Advantages of this framework include: (1) it is axiomatically justified; (2) it is amenable to dynamic and convex optimization; and (3) it is general, in the sense that it captures a full range of risk assessments from risk-neutral to worst case (given the generality of Markov polytopic risk metrics). 

This paper leaves numerous important extensions open for further research. First, we plan to study the case with state and control constraints (preliminary work suggests that such constraints can be readily included in our framework). Second, it is of interest to consider the inclusion of multi-period risk constraints. Third, we plan to characterize the sub-optimal performance of risk-averse MPC by designing lower and upper bounds for the optimal cost. Fourth, we plan to study in detail the offline and hybrid counterparts of the online MPC algorithm presented in this paper. Finally, it is of interest to study the robustness of our approach with respect to perturbations in the system's parameters.

\bibliographystyle{unsrt} 
\bibliography{ref_dyn_mpc}

\end{document}

%% file: preamble.tex
\usepackage{color}
\usepackage{amsmath}
\usepackage{amssymb}
\usepackage{graphicx}
\usepackage{comment}

\newcommand{\real}{{\mathbb{R}}}
\newcommand{\reals}{\real}

\renewcommand{\natural}{{\mathbb{N}}}
\newcommand{\naturals}{\natural}

\newtheorem{theorem}{Theorem}[section]

\newtheorem{lemma}[theorem]{Lemma}

\newtheorem{example}[theorem]{Example}
\newtheorem{definition}[theorem]{Definition}

\newtheorem{assumption}[theorem]{Assumption}

\newcommand{\probnoarg}{\mbox{$\mathbb{P}$}} 
\newcommand{\expectation}[1]{\mbox{$\mathbb{E}\left[#1\right]$}}